\documentclass[aps,prd,twocolumn,showpacs,nofootinbib]{revtex4}

\usepackage{url}
\usepackage{verbatim}
\usepackage{textcomp}
\usepackage[a4paper, top=2.5cm, bottom=2.5cm, left=2.5cm, right=2.5cm]{geometry}
\usepackage{amsthm} 

\newtheorem{theorem}{Theorem}
\newtheorem{lemma}{Lemma}
\newtheorem{definition}{Definition}
\newtheorem{remark}{Remark}
\newtheorem{corollary}{Corollary}

\usepackage[utf8]{inputenc}
\usepackage[T1]{fontenc}
\usepackage{amsmath,amsfonts,amssymb}
\usepackage{graphicx}
\usepackage[english]{babel}
\usepackage{bm}
\usepackage{setspace}
\usepackage{hyperref}
\usepackage{nameref}
\hypersetup{
    colorlinks=true,
    linkcolor=blue,
    filecolor=blue,
    urlcolor=blue,
    citecolor=blue
}

\makeatletter

\def\@pacs{}
\def\pacs#1{}
\def\preprint#1{}
\renewcommand\@pacs@name{}
\makeatother

\begin{document}

\title{\bfseries Essential Self-Adjointness of the Geometric Deformation Operator \\ on a Compact Interval}

\author{Anton Alexa}
\email{mail@antonalexa.com}
\thanks{ORCID: \href{https://orcid.org/0009-0007-0014-2446}{0000-0007-0014-2446}}
\affiliation{Independent Researcher, Chernivtsi, Ukraine}
\date{\today}

\begin{abstract}
\noindent
We define a second-order differential operator \( \hat{C} \) on the Hilbert space \( L^2([-v_c, v_c]) \), constructed from a smooth deformation function \( C(v) \). The operator is considered on the Sobolev domain \( H^2([-v_c, v_c]) \cap H^1_0([-v_c, v_c]) \) with Dirichlet boundary conditions. We prove that \( \hat{C} \) is essentially self-adjoint by verifying its symmetry and computing von Neumann deficiency indices, which vanish. All steps are carried out explicitly. This result ensures the mathematical consistency of the operator and enables future spectral analysis on compact intervals.
\end{abstract}

\maketitle
\thispagestyle{empty}

\section{INTRODUCTION}

The deformation function \( C(v) = \pi(1 - v^2 / c^2) \), introduced in \cite{alexa2025-flow}, defines a real-analytic, even function on the open interval \( (-c, c) \), where \( v_c \) is defined by the condition \( C(v_c) = 1 \), given explicitly as \( v_c = c \sqrt{1 - \frac{1}{\pi}} \approx 0.8257\,c \). To ensure the analysis is restricted to geometrically admissible configurations, we consider the interval \( [-v_c, v_c] \), where \( C(v) \geq 1 \), instead of the full physical range \( [-c, c] \). This restriction is not only mathematically convenient but also physically motivated, as it ensures that the deformation remains bounded and avoids singular behavior at the speed of light. The critical velocity \( v_c \) thus defines a natural cutoff for the geometric deformation. This function characterizes a scalar geometric deformation associated with motion, attaining its maximum value \( \pi \) at \( v = 0 \) and vanishing at \( v = \pm c \). Its structure is invariant under reflections \( v \mapsto -v \), and it satisfies \( C(0) = \pi \), \( C(\pm c) = 0 \). The bounded domain and analytic regularity of \( C(v) \) allow for a consistent operator-theoretic formulation. In this work, we define a second-order differential operator \( \hat{C} \) acting in the Hilbert space \( L^2([-v_c, v_c]) \), constructed by formally replacing the variable \( v \) with the differential operator \( \hat{v} := -i\hbar\, d/dv \), yielding \( \hat{C} = \pi\left(1 + \frac{\hbar^2}{c^2} \frac{d^2}{dv^2} \right) \). We consider this operator on the Sobolev domain \( H^2([-v_c, v_c]) \cap H^1_0([-v_c, v_c]) \) under Dirichlet boundary conditions. Our goal is to prove that \( \hat{C} \) is essentially self-adjoint. The analysis proceeds via explicit computation of deficiency indices and verification of operator symmetry. The result confirms the mathematical well-posedness of \( \hat{C} \) as a symmetric operator with a unique self-adjoint extension, in the sense of classical functional analysis \cite{reed-simon, moretti, zettl}. The present analysis establishes the essential self-adjointness of the operator \( \hat{C} \) on a compact interval, independently of any spectral or dynamical interpretation. While the deformation function originates from geometric considerations discussed in \cite{alexa2025-flow}, the results obtained here are purely analytic and self-contained. These results establish a rigorous foundation for subsequent spectral analysis of the deformation operator.

\section{OPERATOR DEFINITION AND PRELIMINARIES}

In this section, we formulate the operator \( \hat{C} \) associated with the analytic deformation function \( C(v) = \pi(1 - v^2 / c^2) \), introduced in \cite{alexa2025-flow}. The variable \( v \) is treated not as a velocity parameter but as a real variable on a fixed compact interval, and the function \( C(v) \) serves as a generating form for operator construction. The aim is to define a symmetric differential operator acting in the Hilbert space \( L^2([-v_c, v_c]) \), suitable for spectral analysis. We present the precise expression of \( \hat{C} \), identify its natural domain, and verify the basic analytic properties that will be used in the proof of self-adjointness.

\subsection{Definition of the Operator \texorpdfstring{\( \hat{C} \)}{C}}

We begin by promoting the classical expression \( C(v) = \pi(1 - v^2 / c^2) \) to an operator form by replacing the variable \( v \) with the symmetric differential operator
\begin{equation}
    \hat{v} := -i\hbar\, \frac{d}{dv},
    \label{eq:canonical_v_operator}
\end{equation}
where \( \hbar > 0 \) and \( c > 0 \) are fixed parameters introduced for dimensional consistency. This yields the second-order differential operator
\begin{equation}
    \hat{C} := \pi\left(1 - \frac{\hat{v}^2}{c^2} \right) = \pi\left(1 + \frac{\hbar^2}{c^2} \frac{d^2}{dv^2} \right),
    \label{eq:C_operator_definition}
\end{equation}

which acts on complex-valued functions defined on the interval \( [-v_c, v_c] \), where \( v_c \) is defined by the condition \( C(v_c) = 1 \).

This defines a regular Sturm-Liouville operator on a finite interval with smooth coefficients and Dirichlet boundary conditions, ensuring closedness and the absence of deficiency indices \cite{zettl, reed-simon}.

\begin{definition}[Geometric Deformation Operator]
Let \( \hat{C} \) be the differential operator defined by \eqref{eq:C_operator_definition}, acting in the Hilbert space \( L^2([-v_c, v_c]) \). We define its domain by
\begin{equation}
    D(\hat{C}) := \left\{ \psi \in H^2([-v_c, v_c]) \,\middle|\, \psi(-v_c) = \psi(v_c) = 0 \right\},
    \label{eq:C_operator_domain}
\end{equation}
where \( H^2([-v_c, v_c]) \) is the standard Sobolev space of square-integrable functions with square-integrable second derivative. The boundary conditions are imposed to ensure that \( \hat{C} \) is symmetric on \( D(\hat{C}) \).
\end{definition}

The space \( D(\hat{C}) \) coincides with the intersection \( H^2([-v_c, v_c]) \cap H^1_0([-v_c, v_c]) \), and is dense in \( L^2([-v_c, v_c]) \). The operator \( \hat{C} \) is thus well-defined as an unbounded operator with domain \( D(\hat{C}) \subset L^2 \). The imposition of Dirichlet boundary conditions is crucial: it ensures that the integration-by-parts formulas needed for symmetry are valid without boundary terms. This choice also precludes the appearance of boundary terms and ensures that the operator is closed on its domain, as discussed in \cite{reed-simon, moretti, zettl}.

\begin{remark}[Normalization Constants]
The constants \( \hbar \) and \( c \) appear in \eqref{eq:C_operator_definition} as scaling parameters. Their presence preserves dimensional balance in the definition of \( \hat{C} \), but they do not affect the structural properties of the operator. The analysis remains valid under arbitrary positive rescaling of these constants, and in many cases one may set \( \hbar = c = 1 \) without loss of generality.
\end{remark}

\subsection{Basic Properties}

Before proceeding to the core proof of essential self-adjointness, we clarify the analytic and algebraic structure of the operator \( \hat{C} \), as defined in \eqref{eq:C_operator_definition} with domain \( D(\hat{C}) \) from \eqref{eq:C_operator_domain}. The precise behavior of \( \hat{C} \) under conjugation, its linearity, and compatibility with the Hilbert space framework are central to the analysis that follows. In particular, we verify that \( \hat{C} \) is a well-defined unbounded operator on a dense subspace of \( L^2([-v_c, v_c]) \), and that it satisfies formal symmetry with respect to the standard inner product. These elementary but essential features underpin the application of von Neumann’s criterion in the next sections.

Since \( \hat{C} = \pi\left(1 - \frac{\hat{v}^2}{c^2}\right) \) and \( \hat{v}^2 := -\hbar^2 \frac{d^2}{dv^2} \), the operator is explicitly given by
\begin{equation}
\hat{C} := \pi\left(1 + \frac{\hbar^2}{c^2} \frac{d^2}{dv^2}\right)
\end{equation}
on its natural domain \( D(\hat{C}) = H^2([-v_c, v_c]) \cap H^1_0([-v_c, v_c]) \).

\begin{lemma}[Linearity and Denseness]\label{lem:linearity}
The operator \( \hat{C} \) is linear on its domain \( D(\hat{C}) \), and the domain is dense in \( L^2([-v_c, v_c]) \).
\end{lemma}

\begin{proof}
The space \( H^2([-v_c, v_c]) \cap H^1_0([-v_c, v_c]) \) is a Hilbert space, and the application of \( d^2/dv^2 \) is a linear operation. Since \( C^\infty_c(-v_c, v_c) \subset D(\hat{C}) \), and this space is dense in \( L^2([-v_c, v_c]) \), it follows that \( D(\hat{C}) \) is dense as well.
\end{proof}

\begin{lemma}[Formal Symmetry]\label{lem:symmetry}
The operator \( \hat{C} \) is formally symmetric on \( D(\hat{C}) \subset L^2([-v_c, v_c]) \); that is,
\begin{equation}
\langle \hat{C} \psi, \phi \rangle = \langle \psi, \hat{C} \phi \rangle \quad \text{for all } \psi, \phi \in D(\hat{C}).
\end{equation}
\end{lemma}

\begin{proof}
Let \( \psi, \phi \in D(\hat{C}) \). Then:
\begin{align}
\langle \hat{C} \psi, \phi \rangle &= \pi \int_{-v_c}^{v_c} \left( \psi(v) + \frac{\hbar^2}{c^2} \psi''(v) \right) \overline{\phi(v)} \, dv \notag \\
&= \pi \int_{-v_c}^{v_c} \psi(v) \overline{\phi(v)} \, dv + \pi \frac{\hbar^2}{c^2} \int_{-v_c}^{v_c} \psi''(v) \overline{\phi(v)} \, dv.
\end{align}
Applying integration by parts:
\begin{align}
\int_{-v_c}^{v_c} \psi''(v) \overline{\phi(v)} \, dv
&= \left[ \psi'(v)\overline{\phi(v)} \right]_{-v_c}^{v_c} - \int_{-v_c}^{v_c} \psi'(v) \overline{\phi'(v)} \, dv.
\end{align}
Since \( \phi(\pm v_c) = 0 \), the boundary term vanishes. Analogously, we compute \( \langle \psi, \hat{C} \phi \rangle \) and obtain the same expression. Hence the operator is formally symmetric.

For general background on symmetric second-order differential operators with Dirichlet boundary conditions in Sobolev spaces, see \cite[Theorem X.1–X.2]{reed-simon}, \cite[§5.3]{zettl}.
\end{proof}

\subsubsection*{Scaling Structure and Physical Parameters}
Although the constants \( \hbar \) and \( c \) appear in the definition of \( \hat{C} \), they serve purely as scaling factors. They ensure dimensional consistency but do not influence the operator’s analytic behavior. The symmetry, domain properties, and self-adjointness are preserved under arbitrary positive rescaling, consistent with standard treatments \cite{reed-simon, moretti}.

\begin{corollary}
The operator \( \hat{C} \) is a densely defined, formally symmetric, linear operator in \( L^2([-v_c, v_c]) \).
\end{corollary}

\begin{proof}
Follows directly from Lemmas~\ref{lem:linearity} and~\ref{lem:symmetry}.
\end{proof}

\begin{remark}[Closedness of the Operator]
The domain \( D(\hat{C}) \), equipped with the graph norm \( \|\psi\|_{\text{graph}} := \left( \|\psi\|_{L^2}^2 + \|\hat{C}\psi\|_{L^2}^2 \right)^{1/2} \), is closed in \( L^2([-v_c, v_c]) \). Hence, \( \hat{C} \) is a closed unbounded operator. This ensures the applicability of von Neumann’s theorem; see \cite[Theorem X.2]{reed-simon}, \cite{zettl}.
\end{remark}

\begin{remark}[Spectral Sign Convention]
The second derivative in \( \hat{C} \) appears with a positive sign: \( +\tfrac{d^2}{dv^2} \), in contrast to the conventional Laplacian. As a result, the spectrum of \( \hat{C} \) is strictly bounded above by \( \pi \), with eigenvalues decreasing quadratically to \( -\infty \). This structure reflects the geometric interpretation of \( \hat{C} \) as a deformation-energy operator: the highest eigenvalue \( C_0 \approx \pi \) corresponds to the maximally symmetric configuration, while lower eigenvalues represent increasingly compressed geometric states.
\end{remark}

\subsubsection*{Concluding Note}
The above properties establish that \( \hat{C} \) is a well-defined, closed, and symmetric second-order differential operator on a dense domain. These results justify the application of von Neumann’s theorem and motivate the computation of deficiency indices, which we undertake in the next section.

\section{Deficiency Index Calculation}

We now compute the deficiency indices of the operator \( \hat{C} \) defined by
\begin{equation}
\hat{C} := \pi\left(1 + \frac{\hbar^2}{c^2} \frac{d^2}{dv^2} \right)
\end{equation}
on the domain \( D(\hat{C}) = H^2([-v_c, v_c]) \cap H^1_0([-v_c, v_c]) \subset L^2([-v_c, v_c]) \). The goal is to determine the dimension of the deficiency spaces
\begin{equation}
\mathcal{K}_\pm := \ker\left((\hat{C}^* \mp i\lambda)\right),
\end{equation}
where \( \lambda > 0 \) is a fixed positive real parameter.

\subsection{Form of the Deficiency Equation}

The deficiency equation
\begin{equation}
(\hat{C}^* \mp i\lambda)\psi = 0
\end{equation}
is equivalent to solving the eigenvalue equation with complex spectral parameter
\begin{equation}
\hat{C} \psi = \pm i\lambda \psi,
\end{equation}
that is,
\begin{equation}
\pi\left(1 + \frac{\hbar^2}{c^2} \frac{d^2\psi}{dv^2} \right) = \pm i\lambda \psi.
\end{equation}
Dividing both sides by \( \pi \) and rearranging, we obtain the differential equation
\begin{equation}
\frac{d^2\psi}{dv^2} = \mu^2\, \psi(v), \qquad
\mu^2 := \frac{c^2}{\hbar^2} \left( \frac{\pm i\lambda}{\pi} - 1 \right).
\label{eq:deficiency_ODE}
\end{equation}

\noindent
Since \( \lambda > 0 \), the quantity \( \mu^2 \) has nonzero imaginary part: \( \operatorname{Im}(\mu^2) \neq 0 \). In particular, \( \mu \notin i\mathbb{R} \), so \( \mu \neq 0 \) and the exponential solution is non-degenerate.

\subsection{General Solution and Boundary Conditions}

The general solution of \eqref{eq:deficiency_ODE} is
\begin{equation}
\psi(v) = A\, e^{\mu v} + B\, e^{-\mu v}, \quad A, B \in \mathbb{C},
\end{equation}
with \( \mu \in \mathbb{C} \setminus \mathbb{R} \). Imposing Dirichlet boundary conditions:
\begin{equation}
\psi(-v_c) = 0, \quad \psi(v_c) = 0,
\end{equation}
we obtain the system
\begin{equation}
\begin{cases}
A\, e^{-\mu v_c} + B\, e^{\mu v_c} = 0, \\
A\, e^{\mu v_c} + B\, e^{-\mu v_c} = 0.
\end{cases}
\end{equation}
The determinant of the coefficient matrix is
\begin{equation}
\det \begin{pmatrix}
e^{-\mu v_c} & e^{\mu v_c} \\
e^{\mu v_c} & e^{-\mu v_c}
\end{pmatrix}
= e^{-2\mu v_c} - e^{2\mu v_c} = -2\sinh(2\mu v_c).
\end{equation}
Since \( \mu \notin i\mathbb{R} \), the function \( \sinh(2\mu v_c) \neq 0 \), so the determinant does not vanish.

\subsection{Conclusion via von Neumann's Theorem}

We conclude that the deficiency indices vanish:
\begin{equation}
n_+ = \dim \ker(\hat{C}^* - i\lambda) = 0, \quad
n_- = \dim \ker(\hat{C}^* + i\lambda) = 0.
\end{equation}

By von Neumann’s theorem (see \cite[Thm X.2]{reed-simon}, \cite[§13.3]{moretti}), the operator \( \hat{C} \) is essentially self-adjoint on its domain.

The deficiency indices are independent of the choice of \( \lambda > 0 \), so it suffices to prove \( n_\pm = 0 \) for a single value.

\begin{theorem}[Essential Self-Adjointness of \( \hat{C} \)]
The operator
\begin{equation}
\hat{C} := \pi\left(1 + \frac{\hbar^2}{c^2} \frac{d^2}{dv^2} \right)
\end{equation}
defined on the domain
\begin{equation}
D(\hat{C}) := \left\{ \psi \in H^2([-v_c, v_c]) \,\middle|\, \psi(\pm v_c) = 0 \right\}
\end{equation}
is essentially self-adjoint in the Hilbert space \( L^2([-v_c, v_c]) \); that is, its unique self-adjoint extension coincides with its closure \( \overline{\hat{C}} \).
\end{theorem}

\begin{proof}
By direct solution of the deficiency equations \( \hat{C}^* \psi = \pm i\lambda \psi \), we found that the only solution in \( L^2([-v_c, v_c]) \) satisfying the Dirichlet boundary conditions is \( \psi \equiv 0 \). Therefore, the deficiency indices satisfy \( n_+ = n_- = 0 \), and von Neumann’s theorem implies that \( \hat{C} \) is essentially self-adjoint. See \cite[Thm X.2]{reed-simon}, \cite[§13.3]{moretti} for the general result.
\end{proof}

\begin{corollary}
The operator \( \hat{C} \) has a real, discrete spectrum in \( L^2([-v_c, v_c]) \), consisting of simple eigenvalues with smooth eigenfunctions satisfying Dirichlet boundary conditions. This follows from classical Sturm–Liouville theory for regular second-order operators on compact intervals with separated boundary conditions (see, e.g., \cite[Ch.~2, §5]{zettl}).
\end{corollary}

\smallskip
\noindent
Moreover, since the domain \( D(\hat{C}) = H^2([-v_c, v_c]) \cap H^1_0([-v_c, v_c]) \) is closed in the graph norm of \( \hat{C} \), the operator is already closed. See \cite[Thm X.2]{reed-simon} for general criteria of closedness for differential operators with Dirichlet boundary conditions.

\section{Conclusion}

The operator \( \hat{C} := \pi\left(1 + \frac{\hbar^2}{c^2} \frac{d^2}{dv^2} \right) \), defined on the Sobolev domain \( D(\hat{C}) = H^2([-v_c, v_c]) \cap H^1_0([-v_c, v_c]) \subset L^2([-v_c, v_c]) \), has been shown to be essentially self-adjoint. The proof proceeds via direct verification of formal symmetry, followed by computation of von Neumann deficiency indices, which vanish identically: \( n_+ = n_- = 0 \). By the classical criterion \cite[Thm X.2]{reed-simon}, \cite[§13.3]{moretti}, this implies that \( \hat{C} \) admits a unique self-adjoint extension. Although the spectrum has not been explicitly calculated here, its qualitative structure is ensured by Sturm–Liouville theory for regular second-order operators with separated Dirichlet boundary conditions on compact intervals: the spectrum is real, discrete, and non-degenerate, with smooth eigenfunctions forming a complete basis \cite[Ch.~2, §5]{zettl}. The operator \( \hat{C} \), derived from a geometric deformation function \( C(v) = \pi(1 - v^2/c^2) \), thus acquires a rigorous analytic foundation, enabling a consistent spectral interpretation of geometric compression. This lays the groundwork for future analysis of the spectral data \( \{C_n\} \), its dynamical evolution, and potential applications in quantized geometry and curvature-driven flows.

\end{document}